\documentclass[11pt]{article}

\usepackage{amsmath,amssymb,amsthm,mathtools}
\usepackage{pxfonts}
\usepackage[margin=1.15in]{geometry}
\usepackage[colorlinks=true,linkcolor=blue,citecolor=blue,urlcolor=blue]{hyperref}

\newtheorem{theorem}{Theorem}
\newtheorem{proposition}{Proposition}
\theoremstyle{definition}
\newtheorem{assumption}{Assumption}
\theoremstyle{remark}

\newcommand{\E}{\mathsf{E}}
\newcommand{\Pp}{\mathsf{P}}
\newcommand{\R}{\mathcal{R}}

\title{Short-maturity skew stickiness ratio under local volatility}
\author{Masaaki Fukasawa\\[1mm]
\small The University of Osaka}
\date{}

\begin{document}
\maketitle

\begin{abstract}
We prove that the skew stickiness ratio converges to two at short maturity
under local volatility models.  This appears to be the first rigorous proof
of this limit for a general time-dependent local volatility function.  As a
by-product, we strengthen the one-half rule of the implied volatility skew by
removing uniform ellipticity and global bounds on spatial derivatives of
order at least two.  The proof uses a first-order Watanabe expansion.
\end{abstract}

\section{Introduction}

The skew stickiness ratio (SSR), introduced by Bergomi \cite{Bergomi},
measures the infinitesimal response of the at-the-money implied volatility to
a spot move, in units of the at-the-money log-strike skew.  It distinguishes
models with the same static smile but different smile dynamics and measures
the associated cross-gamma risk.

Bergomi identified the value two under local volatility models.  For a local
volatility function independent of calendar time, it follows exactly from a
backward--forward symmetry.  For a general time-dependent local volatility
function, he gave two formal arguments.  They are summarized below after the
notation has been fixed.

Here we prove the short-maturity convergence to two under local smoothness and
positivity, with a global bound only on the first derivative of the local volatility function.  To the best of our knowledge,
this is the first rigorous proof of this limit for a general time-dependent
local volatility function.  
It involves the one-half rule of the implied volatility skew, for which a rigorous proof was given by Al\`os and Garc\'ia-Lorite
\cite[Section~7.6.2]{AlosGarciaLorite} for bounded, uniformly positive local
volatility functions with bounded derivatives.  Our result improves theirs by
removing the uniform positive lower bound and the global bounds on spatial
derivatives of order at least two. 

\section{Model and main result}\label{sec:main}

We work at zero interest and dividend rates.  Let
\begin{equation}\label{eq:SDE}
    \mathrm{d}S_u=a(S_u,u)\,\mathrm{d}B_u,
    \qquad S_t=s>0,
\end{equation}
and denote the relative local volatility by
\begin{equation}\label{eq:relative-local-vol}
    v(s,t)=\frac{a(s,t)}{s}.
\end{equation}
For a maturity $T>t$, let
\begin{equation*}
    P(t,s;K)=\E\big[(K-S_T)_+\mid S_t=s\big].
\end{equation*}
Let $p_{\mathrm{BS}}(s,K,w)$ be the Black--Scholes put price with
total variance $w$, and define the implied total variance $\Sigma$ and
the implied volatility $\widehat v$ by
\begin{equation}\label{eq:implied-variance}
    P(t,s;K)=p_{\mathrm{BS}}\bigl(s,K,\Sigma(t,s;K)\bigr),
    \qquad
    \widehat v(t,s;K)=\sqrt{\frac{\Sigma(t,s;K)}{T-t}}.
\end{equation}
The maturity $T$ is suppressed from the notation on the right-hand side.
The at-the-money log-strike skew is
\begin{equation}\label{eq:skew}
    \mathcal{S}(t,s;T)
      =\left.\frac{\partial}{\partial k}
        \widehat v(t,s;se^k)\right|_{k=0}
      =s\widehat v_K(t,s;s).
\end{equation}

Set $\widehat v_u^S=\widehat v(u,S_u;S_u)$.  The SSR is defined by
\begin{equation}\label{eq:SSR-dynamic}
    \R(t,s;T)
      =\frac{1}{\mathcal{S}(t,s;T)}
        \left.
        \frac{\mathrm{d}\langle\widehat v^S,\log S\rangle_u}
             {\mathrm{d}\langle\log S\rangle_u}
        \right|_{u=t,S_t=s}.
\end{equation}
Since the model is Markovian,
\begin{equation}\label{eq:SSR-Markov}
    \R(t,s;T)
      =\frac{\widehat v_s(t,s;s)+\widehat v_K(t,s;s)}
             {\widehat v_K(t,s;s)}
      =\frac{\Sigma_s(t,s;s)+\Sigma_K(t,s;s)}
             {\Sigma_K(t,s;s)}
\end{equation}
whenever the denominator is nonzero.

\paragraph{Bergomi's arguments.}
When $v$ is independent of calendar time, the backward pricing equation in
the initial spot and the Dupire forward equation in the strike imply the
symmetry
\[
    \widehat v(t,s;K)=\widehat v(t,K;s).
\]
Consequently, $\widehat v_s(t,s;s)=\widehat v_K(t,s;s)$ and the SSR equals two
exactly.

For a general time-dependent local volatility function, Bergomi also obtained
the value two from a first-order perturbation in the log-moneyness slope.  The
expansion is taken around a flat local volatility function at fixed maturity.
Both the numerator and denominator of the SSR vanish in the reference model,
and no remainder estimate uniform in maturity is provided.  Thus the
calculation gives the short-maturity limit of the first-order approximation,
rather than the limit for a fixed model.

His other argument differentiates the pointwise short-maturity harmonic-mean
formula for implied volatility.  Pointwise convergence does not imply
convergence of the spot and strike derivatives entering the SSR.  This
differentiation is therefore also formal
\cite[Sections~2.5.1--2.5.3]{BergomiBook}.

\begin{assumption}\label{ass:regularity}
For some $T_0>0$, $v\in C^\infty((0,\infty)\times[0,T_0])$ and $v>0$.
The solution to \eqref{eq:SDE} started from any positive initial value
remains in $(0,\infty)$, and
\begin{equation}\label{eq:as-bound}
    L:=\sup_{r>0,\,0\leq u\leq T_0}|a_s(r,u)|
      =\sup_{r>0,\,0\leq u\leq T_0}|v(r,u)+r v_s(r,u)|<\infty.
\end{equation}
\end{assumption}

The bound \eqref{eq:as-bound} allows $a$ to be extended to a coefficient on
$\mathbb{R}\times[0,T_0]$ which is globally Lipschitz in the state variable,
uniformly in time, and has linear growth.  Hence \eqref{eq:SDE} admits a
unique global strong solution.  The additional condition in
Assumption~\ref{ass:regularity} is that this solution remains strictly
positive.

\begin{theorem}\label{thm:main}
Under Assumption~\ref{ass:regularity}, for every $t<T_0$ and $s>0$,
\begin{equation}\label{eq:half-rule}
    \lim_{T\downarrow t}\mathcal{S}(t,s;T)
       =\frac12s v_s(s,t).
\end{equation}
If $v_s(s,t)\neq0$, then $\R(t,s;T)$ is well-defined for all sufficiently
small $T-t$ and
\begin{equation}\label{eq:SSR-limit}
    \lim_{T\downarrow t}\R(t,s;T)=2.
\end{equation}
\end{theorem}

\section{Proof}\label{sec:proof}

\subsection{Representation by digital probabilities}

Write $S^r$ for the solution starting from $r$ at time $t$.  Fix
$0<\eta_0<s$.  For $0<|\eta|\leq\eta_0$, put
\[
    \Delta_u^\eta=\frac{S_u^{s+\eta}-S_u^s}{\eta}.
\]
The mean value theorem gives
\[
    \Delta_u^\eta
      =1+\int_t^u A_q^\eta\Delta_q^\eta\,\mathrm{d}B_q,
    \qquad
    A_q^\eta
      =\int_0^1
        a_s\bigl(S_q^s+\theta(S_q^{s+\eta}-S_q^s),q\bigr)
        \,\mathrm{d}\theta .
\]
Thus $|A^\eta|\leq L$.  Since $\Delta^\eta$ is a stochastic exponential,
for every $p>1$,
\begin{equation}\label{eq:flow-Lp}
    \sup_{0<|\eta|\leq\eta_0}\E|\Delta_T^\eta|^p
       \leq
       \exp\left\{\frac{p(p-1)}2L^2(T-t)\right\}.
\end{equation}
Moreover, $\Delta_T^\eta\to J_T$ in $L^p$, where
\begin{equation}\label{eq:J}
    \mathrm{d}J_u=a_s(S_u^s,u)J_u\,\mathrm{d}B_u,
    \qquad J_t=1,
    \qquad
    J_u=\mathcal E\left(\int_t^u a_s(S_q^s,q)\,\mathrm{d}B_q\right).
\end{equation}
Indeed, $S^{s+\eta}\to S^s$ in $L^p(\Omega;C([t,T]))$, hence
$A^\eta\to a_s(S^s,\cdot)$ in probability and in every finite $L^p$ after
integration in time; stability of the linear equations then gives the
claim.

We next note that $S_T^s$ has no atoms.  Choose strictly positive coefficients
$a_n\in C_b^\infty(\mathbb{R}\times[0,T_0])$, bounded away from zero,
which agree with $a$ on $[1/n,n]\times[t,T]$.  Let $S^n$ be the
corresponding solutions driven by the same Brownian motion as $S^s$.
Let $\tau_n$ be the first exit time of $S^s$ from $(1/n,n)$.  The two
solutions agree on $\{\tau_n>T\}$.  If $D$ denotes the
Malliavin--Shigekawa derivative operator and $J^n$ the derivative flow of
$S^n$, then
\[
    D_qS_T^n
      =a_n(S_q^n,q)\frac{J_T^n}{J_q^n},
      \qquad t\leq q\leq T.
\]
All factors on the right are strictly positive.  Hence the
Malliavin covariance of $S_T^n$ is strictly positive.  The
one-dimensional Malliavin--Shigekawa criterion then implies that its law
is absolutely continuous.  Since $S^s$ has positive continuous paths,
$\Pp(\tau_n\leq T)\to0$, and therefore, for every $K$,
\[
    \Pp(S_T^s=K)
       \leq \Pp(\tau_n\leq T)+\Pp(S_T^n=K)
       \longrightarrow0.
\]

For $f(x)=(K-x)_+$,
\[
 \frac{f(S_T^{s+\eta})-f(S_T^s)}{\eta}
   =-\Delta_T^\eta\int_0^1
     \boldsymbol{1}_{\{S_T^s+\theta(S_T^{s+\eta}-S_T^s)<K\}}
     \,\mathrm{d}\theta .
\]
The absolute value is bounded by $|\Delta_T^\eta|$.  The quotient
converges in probability to
$-J_T\boldsymbol{1}_{\{S_T^s<K\}}$.  The bound
\eqref{eq:flow-Lp} gives uniform integrability, so
\begin{equation}\label{eq:put-delta-flow}
    P_s(t,s;K)
      =-\E\bigl[J_T\boldsymbol{1}_{\{S_T^s<K\}}\bigr].
\end{equation}

Boundedness of $a_s$ implies that $J$ is a true martingale.  Define
$\mathrm{d}\mathsf Q=J_T\,\mathrm{d}\Pp$ on $\mathcal F_T$.  Under
$\mathsf Q$,
\[
    \widetilde B_u
       =B_u-\int_t^u a_s(S_q^s,q)\,\mathrm{d}q
\]
is Brownian, and $S^s$ satisfies
\begin{equation}\label{eq:tilde-SDE}
    \mathrm{d}\widetilde S_u
      =a(\widetilde S_u,u)a_s(\widetilde S_u,u)\,\mathrm{d}u
       +a(\widetilde S_u,u)\,\mathrm{d}\widetilde B_u,
    \qquad \widetilde S_t=s.
\end{equation}
We write $\widetilde S$ for a process with this law under $\Pp$.
Local smoothness gives uniqueness in law up to the first exit from compact
subsets of $(0,\infty)$, which is sufficient here.  Equation
\eqref{eq:put-delta-flow} becomes
\begin{equation}\label{eq:q-digital}
    P_s(t,s;K)
       =-\Pp(\widetilde S_T<K\mid\widetilde S_t=s).
\end{equation}

Write $w=\Sigma(t,s;s)$.  Local ellipticity implies $w>0$.  Let $\Phi$
and $\phi$ denote the standard normal distribution function and density.  At
the money,
\begin{equation}\label{eq:BS-derivatives}
\begin{split}
    (p_{\mathrm{BS}})_s(s,s,w)&=-\Phi(-\sqrt{w}/2),\\
    (p_{\mathrm{BS}})_K(s,s,w)&= \Phi( \sqrt{w}/2),\\
    (p_{\mathrm{BS}})_w(s,s,w)&=
       \frac{s\phi(\sqrt{w}/2)}{2\sqrt{w}}.
\end{split}
\end{equation}
Define
\begin{align}
    Y(t,s;T)
      &=\Pp(S_T<s\mid S_t=s)-\Phi(\sqrt{w}/2), \label{eq:Y}\\
    X(t,s;T)
      &=Y(t,s;T)+\Phi(-\sqrt{w}/2)
        -\Pp(\widetilde S_T<s\mid\widetilde S_t=s). \label{eq:X}
\end{align}
The absence of atoms also gives
$P_K(t,s;K)=\Pp(S_T<K\mid S_t=s)$.  Since
$(p_{\mathrm{BS}})_w>0$, implicit differentiation of
\eqref{eq:implied-variance} gives
\begin{equation}\label{eq:XY-Sigma}
    Y=(p_{\mathrm{BS}})_w\Sigma_K,
    \qquad
    X=(p_{\mathrm{BS}})_w(\Sigma_s+\Sigma_K).
\end{equation}
Consequently,
\begin{equation}\label{eq:R-XY}
    \R(t,s;T)=\frac{X(t,s;T)}{Y(t,s;T)}.
\end{equation}
Moreover,
\begin{equation}\label{eq:Y-skew}
    Y(t,s;T)
      =\sqrt{T-t}\,\phi(\sqrt{w}/2)\mathcal{S}(t,s;T).
\end{equation}

\subsection{Localization}

Fix $(t,s)$ and choose $\rho,\delta>0$ such that
\[
    Q=[s-2\rho,s+2\rho]\times[t,t+\delta]
       \subset(0,\infty)\times[0,T_0].
\]
Set
\[
    A=\sup_Q|a|,
    \qquad C=\sup_Q|a a_s|.
\]
Let $\tau$ and $\widetilde\tau$ be the first exit times of $S$ and
$\widetilde S$ from $(s-\rho,s+\rho)$, and put $h=T-t\leq\delta$.
For the stopped martingale part $N$ of $S-s$, the
Burkholder--Davis--Gundy inequality gives, for every $p\geq2$,
\[
 \E\left[\sup_{t\leq u\leq t+h}|N_u|^p\right]
    \leq c_p A^p h^{p/2}.
\]
An exit of $S$ requires $\sup|N|\geq\rho$.  For $\widetilde S$, the
finite-variation part before exit has absolute value at most $Ch$.
Consequently, if $Ch\leq\rho/2$,
\begin{equation}\label{eq:exit-bound}
 \Pp(\tau\leq t+h)+\Pp(\widetilde\tau\leq t+h)
    \leq c_p h^{p/2}.
\end{equation}

Choose a strictly positive coefficient
$\bar a\in C_b^\infty(\mathbb{R}\times[0,T_0])$ which is bounded away
from zero and agrees with $a$ on $Q$.  Then
$\bar a\,\bar a_s=a a_s$ on $Q$.  Let $\bar S$ and
$\bar{\widetilde S}$ solve the corresponding localized original and
auxiliary equations, driven by the same Brownian motions as $S$ and
$\widetilde S$.  Pathwise uniqueness implies equality up to the respective
exit times.

Introduce the three quantities to which localization will be applied:
\[
\begin{aligned}
 d_h&=\Pp(S_{t+h}<s\mid S_t=s),\\
 \widetilde d_h&=\Pp(\widetilde S_{t+h}<s\mid\widetilde S_t=s),\\
 p_h&=\E[(s-S_{t+h})_+\mid S_t=s].
\end{aligned}
\]
and let $\bar d_h,\bar{\widetilde d}_h,\bar p_h$ denote their localized
counterparts.  The coupling gives, without any regularity requirement on
the digital payoff,
\[
 |d_h-\bar d_h|\leq\Pp(\tau\leq t+h),
 \qquad
 |\widetilde d_h-\bar{\widetilde d}_h|
       \leq\Pp(\widetilde\tau\leq t+h).
\]
Taking $p=2$ in \eqref{eq:exit-bound}, both differences are
$O(h)=o(\sqrt h)$.

For the put payoff, equality before exit and its Lipschitz continuity give
\[
 |p_h-\bar p_h|
 \leq
 \E\left[(|S_{t+h}|+|\bar S_{t+h}|)
          \boldsymbol{1}_{\{\tau\leq t+h\}}\right].
\]
The second moments of $S_{t+h}$ and $\bar S_{t+h}$ are bounded uniformly
for small $h$.  By Cauchy--Schwarz and \eqref{eq:exit-bound} with $p=4$,
the right-hand side is $O(h)$.  We have therefore proved
\begin{equation}\label{eq:localization-transfer}
 d_h=\bar d_h+o(\sqrt h),\qquad
 \widetilde d_h=\bar{\widetilde d}_h+o(\sqrt h),\qquad
 p_h=\bar p_h+o(\sqrt h).
\end{equation}
These are the only quantities transferred from the localized model.
In particular, no asymptotic formula for an implied-volatility derivative
is differentiated or localized: the exact identities
\eqref{eq:XY-Sigma} reduce the problem to
\eqref{eq:localization-transfer} before localization is used.

\subsection{Small-time expansion}

Put
\[
    \varepsilon=\sqrt{T-t},\qquad
    a_0=a(s,t),\qquad a_1=a_s(s,t),\qquad v_0=\frac{a_0}{s}.
\]
Since $\bar a=a$ on $Q$,
\[
    \bar a(s,t)=a_0,\qquad \bar a_s(s,t)=a_1.
\]
For $0\leq u\leq1$, define
\[
    \bar S_u^\varepsilon=\bar S_{t+\varepsilon^2u},\qquad
    \bar{\widetilde S}_u^\varepsilon
       =\bar{\widetilde S}_{t+\varepsilon^2u},\qquad
    W_u=\frac{B_{t+\varepsilon^2u}-B_t}{\varepsilon},
\]
where $B$ denotes the Brownian motion driving the corresponding equation.
The stochastic Taylor expansions for the localized equations hold in
every Malliavin--Shigekawa Sobolev space:
\begin{align}
    \bar S_1^\varepsilon
      &=s+\varepsilon a_0W_1
        +\frac{\varepsilon^2a_0a_1}{2}(W_1^2-1)
        +O_{\mathbb{D}^\infty}(\varepsilon^3),
        \label{eq:S-expansion}\\
    \bar{\widetilde S}_1^\varepsilon
      &=s+\varepsilon a_0W_1
        +\frac{\varepsilon^2a_0a_1}{2}(W_1^2+1)
        +O_{\mathbb{D}^\infty}(\varepsilon^3).
        \label{eq:tilde-expansion}
\end{align}
The extra constant in \eqref{eq:tilde-expansion} is the contribution of
the drift $a a_s$.

Define
\[
    \bar F^\varepsilon
       =\frac{\bar S_1^\varepsilon-s}{\varepsilon a_0},
    \qquad
    \bar{\widetilde F}^\varepsilon
       =\frac{\bar{\widetilde S}_1^\varepsilon-s}
              {\varepsilon a_0}.
\]
Then
\[
    \bar F^\varepsilon
       =W_1+\varepsilon G
                    +O_{\mathbb{D}^\infty}(\varepsilon^2),
    \qquad
    \bar{\widetilde F}^\varepsilon
       =W_1+\varepsilon\widetilde G
                    +O_{\mathbb{D}^\infty}(\varepsilon^2),
\]
where
\[
    G=\frac{a_1}{2}(W_1^2-1),
    \qquad
    \widetilde G=\frac{a_1}{2}(W_1^2+1).
\]
Both families are uniformly Malliavin nondegenerate.  Indeed, let
$\underline a>0$ be a lower bound for $\bar a$ and let
$\mathcal J_{r,1}^\varepsilon$ denote the derivative flow from time $r$
to time $1$.  For the first family,
\[
   D_r\bar F^\varepsilon
      =\frac{\bar a(\bar S_r^\varepsilon,t+\varepsilon^2r)}{a_0}
         \mathcal J_{r,1}^\varepsilon,
\]
and the analogous formula holds for $\bar{\widetilde F}^\varepsilon$.
Thus
\[
   \gamma_{\bar F^\varepsilon}
      \geq \frac{\underline a^2}{a_0^2}
          \inf_{0\leq r\leq1}
             |\mathcal J_{r,1}^\varepsilon|^2.
\]
The coefficients of the rescaled derivative-flow equations are bounded
uniformly in $\varepsilon$.  The flows and their inverses therefore have
uniform maximal moments of every order, which gives uniform inverse
moments of both Malliavin covariances.

Proposition~\ref{prop:Watanabe-digital} gives
\[
    \bar d_h
      =\frac12-\varepsilon\E[G\delta_0(W_1)]+o(\varepsilon),
    \qquad
    \bar{\widetilde d}_h
      =\frac12-\varepsilon\E[\widetilde G\delta_0(W_1)]
         +o(\varepsilon).
\]
Since
\[
    \E[G\delta_0(W_1)]=-\frac{a_1}{2}\phi(0),
    \qquad
    \E[\widetilde G\delta_0(W_1)]=\frac{a_1}{2}\phi(0),
\]
we obtain
\[
    \bar d_h
      =\frac12+\frac{a_1}{2}\phi(0)\sqrt h+o(\sqrt h),
    \qquad
    \bar{\widetilde d}_h
      =\frac12-\frac{a_1}{2}\phi(0)\sqrt h+o(\sqrt h).
\]
Moreover, \eqref{eq:S-expansion} and the Lipschitz continuity of the put
payoff give
\[
    \bar p_h=a_0\phi(0)\sqrt h+o(\sqrt h).
\]
Combining these three localized expansions with
\eqref{eq:localization-transfer} yields, for the original model,
\begin{align}
    \Pp(S_T<s\mid S_t=s)
      &=\frac12+\frac{a_1}{2}\phi(0)\sqrt{T-t}
        +o(\sqrt{T-t}), \label{eq:S-digital}\\
    \Pp(\widetilde S_T<s\mid\widetilde S_t=s)
      &=\frac12-\frac{a_1}{2}\phi(0)\sqrt{T-t}
        +o(\sqrt{T-t}), \label{eq:tilde-digital}
\end{align}
and
\begin{equation}\label{eq:put-expansion}
    P(t,s;s)=a_0\phi(0)\sqrt{T-t}+o(\sqrt{T-t}).
\end{equation}

Since
\[
    p_{\mathrm{BS}}(s,s,w)
      =s\bigl(2\Phi(\sqrt{w}/2)-1\bigr)
      =s\phi(0)\sqrt{w}+O(w^{3/2}),
\]
\eqref{eq:put-expansion} first implies $w\to0$ and then
\begin{equation}\label{eq:w-expansion}
    \sqrt{w}=v_0\sqrt{T-t}+o(\sqrt{T-t}).
\end{equation}

\subsection{Conclusion}

Equations \eqref{eq:Y}, \eqref{eq:S-digital}, and
\eqref{eq:w-expansion} imply
\begin{align}
    Y(t,s;T)
      &=\frac{\phi(0)}{2}(a_1-v_0)\sqrt{T-t}
        +o(\sqrt{T-t}) \notag\\
      &=\frac{\phi(0)}{2}s v_s(s,t)\sqrt{T-t}
        +o(\sqrt{T-t}). \label{eq:Y-expansion}
\end{align}
Similarly,
\begin{equation}\label{eq:X-expansion}
    X(t,s;T)
      =\phi(0)s v_s(s,t)\sqrt{T-t}
        +o(\sqrt{T-t}).
\end{equation}
The one-half rule follows from \eqref{eq:Y-skew} and
\eqref{eq:Y-expansion}.  If $v_s(s,t)\neq0$, the SSR limit follows from
\eqref{eq:R-XY}, \eqref{eq:Y-expansion}, and \eqref{eq:X-expansion}.
\qed

\appendix

\section{Watanabe expansion for digital probabilities}
\label{app:Watanabe}

We use Theorems~2.1 and~2.3 of Watanabe
\cite{Watanabe}.  Theorem~2.1 defines the pullback $\mathcal T(F)$ of a
tempered distribution $\mathcal T$ by a nondegenerate Wiener functional
$F$.  Theorem~2.3 gives an asymptotic expansion of these pullbacks for a
uniformly nondegenerate family.  The proposition below is its
one-dimensional first-order specialization with $\mathcal T=H_-$.

For a Wiener functional $F$, write
\[
    \gamma_F=\int_0^1|D_uF|^2\,\mathrm{d}u
\]
for its Malliavin covariance, where $D$ is the Malliavin--Shigekawa
derivative operator.  Write
$\mathbb{D}^\infty=\bigcap_{k,p}\mathbb{D}^{k,p}$ and let
$H_-(x)=\boldsymbol{1}_{\{x<0\}}$, regarded as a tempered distribution.
Under nondegeneracy, Theorem~2.1 of \cite{Watanabe} defines both
$H_-(F)$ and $\delta_0(F)$ as generalized Wiener functionals.

\begin{proposition}[First-order digital expansion]
\label{prop:Watanabe-digital}
Let $F^\varepsilon,F_0,F_1\in\mathbb{D}^\infty$ and suppose, in every
$\mathbb{D}^{k,p}$,
\[
    F^\varepsilon=F_0+\varepsilon F_1+o(\varepsilon).
\]
Assume that for every $p>1$, there are $\varepsilon_0>0$ and $C_p<\infty$
such that
\[
    \sup_{0\leq\varepsilon\leq\varepsilon_0}
       \E[\gamma_{F^\varepsilon}^{-p}]\leq C_p.
\]
Then, in the space of generalized Wiener functionals,
\[
    H_-(F^\varepsilon)
      =H_-(F_0)-\varepsilon F_1\delta_0(F_0)+o(\varepsilon).
\]
Consequently,
\[
    \Pp(F^\varepsilon<0)
       =\Pp(F_0<0)-\varepsilon\E[F_1\delta_0(F_0)]+o(\varepsilon).
\]
\end{proposition}

\begin{proof}
Apply Theorem~2.3 of \cite{Watanabe}, truncated after the first-order
term, to the tempered distribution $H_-$.  The coefficients are those
of the formal distributional Taylor expansion
\[
 H_-(F_0+\varepsilon F_1)
   =H_-(F_0)+\varepsilon F_1H_-'(F_0)+o(\varepsilon).
\]
Since $H_-'=-\delta_0$, this gives the asserted expansion in the space of
generalized Wiener functionals.  Pairing it with the constant test
functional $1\in\mathbb{D}^\infty$ gives the probability formula.
\end{proof}

\section*{Declaration on the use of generative AI}

ChatGPT (OpenAI) was used to draft and edit the exposition and prepare the
\LaTeX{} manuscript.  The author reviewed all AI-assisted material and takes
full responsibility for the content.

\end{document}